\documentclass[twoside,10.5pt]{article}
\usepackage{mathrsfs}
\usepackage{pifont}
\usepackage{amsmath}
\usepackage{amsthm}
\usepackage{txfonts}
\usepackage{geometry}
\usepackage{latexsym}
\usepackage{amssymb}
\usepackage{graphicx}
\usepackage{geometry}
\usepackage{xcolor} 
 \newcommand{\eps}{\varepsilon}

 \newcommand{\N}{\mathbb{N}}

 \newcommand{\prob}{\mathbb{P}}

 \newcommand{\me}{\mathbb{E}}
 
 \renewcommand{\P}{\mathbb{P}}
 
 \newcommand{\one}{\1}

 \newcommand{\skric}{{\mathcal C}}

 \newcommand{\skrie}{{\mathcal E}}
 
 \newcommand{\skrig}{{\mathcal G}}

 \newcommand{\skrik}{{\mathcal K}}
 \newcommand{\skril}{{\mathcal L}}
 
 \newcommand{\skriw}{{\mathcal W}}

 \newcommand{\skris}{{\mathcal S}}
 
 \newcommand{\skriv}{{\mathcal V}}
 \newcommand{\skrix}{{\mathcal X}}
 \newcommand{\skriy}{{\mathcal Y}}
 
 \newcommand{\heap}[2]{\genfrac{}{}{0pt}{}{#1}{#2}}
 \newcommand{\sfrac}[2]{\mbox{$\frac{#1}{#2}$}}

\def\1{{\mathchoice {1\mskip-4mu\mathrm l}      
{1\mskip-4mu\mathrm l} {1\mskip-4.5mu\mathrm l} {1\mskip-5mu\mathrm
l}}}

\newcommand{\eq}{\begin{equation}}
\newcommand{\en}{\end{equation}}

\newenvironment{Proof}
{\vskip0.1cm\noindent{\bf Proof. }{\hspace*{0.3cm}}}%
{\nopagebreak {\hspace*{\fill}\rule{2mm}{2mm}}\\ }

{\nopagebreak {\hspace*{\fill}\rule{2mm}{2mm}}\\ }

\renewcommand{\subsection}{\secdef \subsct\sbsect}
\newcommand{\subsct}[2][default]{\refstepcounter{subsection}
\vspace{0.15cm} {\flushleft\bf
\arabic{section}.\arabic{subsection}~\bf #1  }
\nopagebreak\nopagebreak}
\newcommand{\sbsect}[1]{\vspace{0.1cm}\noindent
{\bf #1}\vspace{0.1cm}}

\newtheorem{theorem}{Theorem}[section]
\newtheorem{lemma}[theorem]{Lemma}

\newtheoremstyle{thm}{1.5ex}{1.5ex}{\itshape\rmfamily}{}
{\bfseries\rmfamily}{}{2ex}{}

\newtheoremstyle{rem}{1.3ex}{1.3ex}{\rmfamily}{}
{\itshape\rmfamily}{}{1.5ex}{} \theoremstyle{rem}

\refstepcounter{subsection}

\def\thebibliography#1{\section*{References}
  \list%
  {\arabic{enumi}.}
    {\settowidth\labelwidth{[#1]}\leftmargin\labelwidth
    \advance\leftmargin\labelsep
    \parsep0pt\itemsep0pt
    \usecounter{enumi}}
    \def\newblock{\hskip .11em plus .33em minus .07em}
    \sloppy                   
    \sfcode`\.=1000\relax}

\begin{document}
\begin{center}
{\LARGE{ Exponential approximation, method of  types for empirical
neighbourhood distributions of  random graphs  by  random
allocations }}\\[20pt]
\end{center}
\centerline{\sc{By Kwabena Doku-Amponsah}}
\renewcommand{\thefootnote}{}

\footnote{University  of   Ghana, Statistics Department, P. O. Box
lg 115, Accra,kdoku@ug.edu.gh}
\renewcommand{\thefootnote}{}


\textbf{Abstract} In  this  article we find  exponential good
approximation of the empirical neigbourhood distribution of
symbolled random graphs conditioned to a given  empirical symbol
distribution and empirical pair distribution.  Using this
approximation  we  shorten or simplify the proof of (Doku-Amponsah
and Morters 2010,  Theorem~2.5); the large deviation principle (LDP)
for empirical neigbourhood distribution of symbolled random graphs.
We also show that the LDP for the empirical degree measure of the
classical Erd\H{o}s-R\'{e}nyi graph is a special case of
(Doku-Amponsah and Moerters, 2010,  Theorem~2.5). From the LDP for
the empirical degree measure, we derive an LDP for the the
proportion of isolated vertices in the classical Erd\H{o}s-R\'{e}nyi
graph.


\textbf{1. Introduction}

The  Erd\H{o}s-R\'{e}nyi graph $\skrig(n,p)$ or $\skrig(n,nc/2)$ is
the simplest imaginable random graph, which  arises  by  taking  $n$
vertices, and placing an edge between  any  two of  distinct nodes
or  vertices with a fixed probability $0<p<1$ or  inserting  a fixed
number $nc$ edges at random among the $n$ vertices.  See,~(Van Der
Hofstad ,2009). Several large deviation (LD) results for this graphs
have been found. See, for example  ( O'Connell ,1998), (Biggins and
Penman, 2009), (Doku-Amponsah and Moerters, 2010), ( Doku-Amponsah,
2006), (Bordenave and Caputo, 2013), (Mukherjee, 2013).

( O'Connell ,1998) proved an LDP  for    the  relative  size of  the
largest connected component  and  the  number  of
 isolated  vertices  in  the random  graph  $\skrig(n,p)$  with  $p=O(1/n).$ (O'Connell ,1998) also presented  an  LDP
 and a
  related result for  the  number of  isolated  vertices in   the  random  graph
  $\skrig(n,c/n).$ i.e. the  near-critical  or  sparse  case.
  An  LDP  for  the
 proportion  of  edges  to   the  number  of  potential vertices  of
 the  supercritical case  has  been  found  by
 (Biggins  and  Penman, 2009). (Doku-Amponsah and Moerters, 2010),
( Doku-Amponsah, 2006)  or ( Doku-Amponsah, 2012)
  obtained  several LDPs , including  the  LDP for
empirical degree distribution    for  near-critical  or
 sparse  case.
   (Boucheron  et  al., 2002)  attempted  to
 prove from  the  LDP  for  the  empirical  occupancy  process
 an   LDP  for  the  degree  distribution  of  the random  graph
 $\skrig(n,nc).$ But
(Doku-Amponsah et  al., 2010)  conjectured  that  the  prove of this
LDP does  not hold. Recently, (Bordenave and Caputo, 2013) obtained
LDPs for the empirical  neighbourhood distribution  in
$\skrig(n,c/n)$ and   $\skrig(n,nc/2).$ The LDP for the  empirical
degree distribution in $\skrig(n,c/n)$ has been proved   in
(Mukherjee, 2013).

Our main aim in this article is  to  obtain an  exponential
approximation result, see Lemma~\ref{exp}, for the empirical
Neighbourhood distribution of symbolled random graphs.  Using this
result we shorten or simplify the proof of the LDP for empirical
Neighbourhood distribution of symbolled random graphs conditioned on
a  given empirical symbol measure and empirical pair distribution.
See example, (Doku-Amponsah et  al., 2010) or(
Doku-Amponsah,Theorem~2.5.1, 2012).

 Further,
we show that the large deviation principle for the empirical degree
measure of $\skrig(n,nc/2)$ is  a special case  of (
Doku-Amponsah,2012, Theorem~2.5.1). From  this result we find an LDP
for the proportion  of  isolated vertices  in the graph
$\skrig(n,nc/2).$ Note that the LDP for  the  proportion  of
isolated  vertices in the graph $\skrig(n,nc/2)$ is new in the
literature. See, ( O'Connell ,1998) for similar result.

The  main  technique  used  in  this  article  is the  method of
 types,  see  ( Dembo  and Zeitouni, 1998, Section~2.1). The  method  of  types  is applied  to
  an exponential approximate  model for  the  symbolled  random  graph model  which
 we shall  obtain by  randomly allocating symbolled balls in to symbolled  bins.

The symbol  random  graphs, see (Penman, 1998) or Inhomogeneous
random graphs, see  (Van Der Hofstad ,2009), which has
Erd\H{os}-R\'{e}nyi graph with one symbol   as an example permit a
dependence between symbol connectivity of  the nodes.In next
subsection, we review the \emph{symbolled random graph model} as in
(Doku-Amponsah et  al., 2010).

\emph{ 1.1 The symbolled random  graph model}\label{SNS}

We begin by fixing the following notations. Let $\skrix$ be a symbol
or colour set. Let $\skriv=\{1,\ldots,n\}$ be a fixed set of $n$
vertices. Denote by $E$ the edge set.i.e.
$$E\subset\skrie:=\big\{(e_1,e_2)\in \skriv\times \skriv \, : \, e_1<e_2\big\},$$ where we  have  used  the \emph{formal} ordering of links  to simply
describe \emph{unoriented} edges.

Let $p_n\colon\skriy\times\skriy\rightarrow [0,1]$ be a symmetric
function and  $\mu:\skriy\to[0,1]$ a probability law.  We may
describe the simply \emph{symbolled random graph}~$Y$ with $n$
vertices in  the following
 manner: Any node $v\in \skriv$  gets symbol $Y(v)$ independently
 and  identically
according to the {\em symbol law} $\lambda.$  Given the colours, we
join any two nodes $u,v\in \skriv$ with  an edge, independently of
everything else, with a {\em edge probability} $p_n(Y(u),Y(v))$
otherwise  we keep them disconnected.  We always look at
$Y=((Y(v)\,:\,v\in V),E)$ under the combine distribution of graph
and symbol. We interpret $Y$ as symbolled random graph and consider
$Y(v)$ as the symbol of the node $v$. See,  Cannings and Penman,
2003) or (Penman, 1998).

Our interest in  this article is  on  the symbolled random graph
models in the \emph{near-critical} critical regimes. Thus, we look
at  cases when the edge probability $p_n(a,b)$ satisfies
$np_n(a,b)\to \skric(a,b),$ for all $a,b\in\skriy,$ and
$\skric\colon\skriy\times\skriy \to[0,\,\infty).$

By $\skriw(\skrix)$ we denote the space of probability measures  on
a finite or countable set $\skrix$, and by $\tilde\skriw(\skrix)$ we
denote the subspace of finite measures   defined on $\skrix,$ and we
endow both with the weak topology. Further, we denote by
$\tilde\skriw_*(\skrix \times \skrix)$ the subspace of symmetric
measures. By convection, we write $$\N=\{0,1,2,...\}.$$

For any symbolled graph $Y=((Y(v)\,:\,v\in \skriv),E)$ with $n$
nodes we recall from (Doku-Amponsah and Moerters, 2010), the
definition of  the \emph{empirical symbol
distribution}~$L_Y^1\in\skriw(\skriy)$,~by
$$L_Y^1(a):=\frac{1}{n}\sum_{u\in V}\delta_{Y(u)}(a),\quad\mbox{ for $a\in\skriy$, }$$
and  the \emph{empirical pair distribution}
$L_Y^2\in\tilde\skriw_*(\skriy\times\skriy),$ by
$$L_Y^2(b,a):=\frac{1}{n}\sum_{(e_1,e_2)\in E}[\delta_{(Y(e_2),\,Y(e_1))}+
\delta_{(Y(e_1),\,Y(e_2))}](b,a),\quad\mbox{ for $a,b\in\skriy$. }$$
We observe that by  definition $L_Y^2$ is  finite symmetric measure
with total mass $\|L_Y^2\|$ equal to $2|E|/n$. Finally we recall the
definition of   the \emph{empirical Neighbourhood distribution}
$M_Y\in\skriw(\skriy\times\N^{\skriy})$, by
$$M_Y(a,\ell):=\frac{1}{n}\sum_{u\in V}\delta_{(Y(u),\skril(u))}(a,l),\quad
\mbox{ for $(a,l)\in\skriy\times\N^{\skriy}$, }$$ where
$\skril(u)=(\ell^{u}(b),\,b\in\skriy)$ and $\ell^{u}(b)$ is the
number of nodes of symbol $b$ linked to node $v$. For any
$\mu\in\skriw(\skriy\times\N^{\skriy})$we  denote by $\mu_1$  the
$\skriy-$ marginal of $\mu$ and for  every
$(b,a)\in\skriy\times\skriy,$ let $\mu_2$  be the law  of  the pair
$(a,l(b))$ under  the  measure $\mu.$ Define the  measure (finite),
$\langle\mu(\cdot,\ell),\,l(\cdot)\rangle\in\tilde\skriw(\skriy\times\skriy)$
by
$$\Delta_2(\mu)(b,a):=
\sum_{l(b)\in\N}\mu_2(a,l(b))l(b), \quad\mbox{ for $a,b\in\skriy$}$$
and  write $\Delta_1(\mu)=\mu_1.$ We define the function  $\Delta
\colon \skriw(\skriy\times\N^{\skriy}) \to \skriw(\skriy) \times
\tilde\skriw(\skriy \times \skriy)$ by
$\Delta(\mu)=(\Delta_1(\mu),\Delta_2(\mu))$ and note that
$\Delta(M_X)=(L_Y^1, L_Y^2).$ Observe  that $\Delta_1$ is a
continuous function  but $\Delta_2$ is \emph{discontinuous} in the
weak topology. In particular, in  the summation  $\displaystyle
\sum_{l(b)\in\N}\mu_2(a,l(b))l(b)$ the function $l(b)$ may be
unbounded and  so  the  functional $\displaystyle
\mu\to\Delta_2(\mu)$ would not be continuous in the weak topology.
We call a pair of measures
$(\pi,\mu)\in\tilde{\skriw}(\skriy\times\skriy)\times\skriw(\skriy\times\N^{\skriy})$
\emph{sub-consistent} if
\begin{equation}\label{consistent}
\Delta_2(\mu)(b,a) \le \pi(b,a), \quad\mbox{ for all
$a,b\in\skriy,$}
\end{equation}
and \emph{consistent} if equality holds in \eqref{consistent}.   For
any $n\in\N$ we define the  following  sets:
$$\begin{aligned}
\skriw_n(\skriy) & := \big\{ \nu\in \skriw(\skriy) \, : \, n\nu(b) \in \N \mbox{ for all } b\in\skriy\big\},\\
\tilde \skriw_n(\skriy \times \skriy) & := \big\{ \pi\in
\tilde\skriw_*(\skriy\times\skriy) \, : \, \sfrac
n{1+\one\{a=b\}}\,\pi(b,a) \in \N  \mbox{ for all } b,a\in\skriy
\big\}\,.
\end{aligned}$$

\emph{1.2 The conditional symbolled random graph models}

In  the remaining part  of  this  article we may assume that
$\nu(a)>0$ for all $a\in\skriy$. Note that the law of the symbolled
random graph given the empirical symbol distribution $\nu_n$ and
empirical pair distribution ~$\pi_n$,
$$\prob_{(\nu_n,\pi_n)}=\prob\{ \,\cdot\,  \,|\,\Delta(M_Y)=(\nu_n,\pi_n)\},$$
may be described as  follows:

\begin{itemize}
\item We assign symbols to the nodes by drawing  from the pool
of $n$~symbols which contains any symbol $a\in\skriy,$  $n\nu_n(a)$
times without replacement;

\item for each unordered pair $\{a,b\}$ of symbols  we construct (exactly) $m_n(b,a)$ edges by
drawing without replacement from the collection of potential edges
linking nodes of symbol $a$ and $b$,
where $$m_n(b,a):=\left\{ \begin{array}{ll} n\, \pi_n(b,a) & \mbox{if $a\not=b$}\\
\frac n2\, \pi_n(b,a) & \mbox{if }a=b\, .\end{array}\right.$$
\end{itemize}

By $Y_n$  we  denote  the conditional symbolled random graph   with
empirical symbol measure $\nu_n$  and  empirical pair
measure~$\pi_n.$

\textbf{2. Main  Results}

The  main  theorem  in  this  section   is  an  LDP  for  the
proportion  of  isolated  nodes in the  random  graph
$\skrig(n,nc/2).$ We recall  from  (Doku-Amponsah and Moerters,
2010), the \emph{empirical degree measure} $D_Y \in \skriw(\N )$ of
the symbolled random as
$$D_Y(k)= \sum_{b\in\skriy}\sum_ {l\in\N^{\skriy}} \delta_k\big({\textstyle
\sum_{a\in\skriy}} l(a)\big)\, M_Y(b,l), \qquad \mbox{ for
$k\in\N$.}$$

 \begin{theorem}\label{main3}
Suppose $D_Y$ is the empirical degree measure of the  random graph
$\skrig(n,nc/2).$ Then,  as $n\to\infty,$ the  proportion  of
isolated  nodes  $D_Y(0),$ obeys  an  LDP with good, convex rate
function

\begin{align}\label{randomg.isolated1}
\eta(x)= \left\{ \begin{array}{ll}
x\log\sfrac{x}{e^{-c}}+(1-x)\log\sfrac{(1-x)}{(1-e^{-c})}+c\log\lambda-c\log
c, & \mbox{if\,
$x\ge 1-c $,}\\
\infty & \mbox{ if  $ x< 1-c $,}
\end{array}\right.
\end{align}
where  $\lambda=\lambda(x,c)$  is  the  unique  root of
$\sfrac{1-e^{-\lambda}}{\lambda}=\sfrac{1-x}{c}.$
\end{theorem}

Theorem~\ref{main2}  below  and  the contraction  principle imply
the  LDP for  the  proportion of  isolated vertices.i.e.
Theorem~\ref{main3}. (O'Connell,  1998)  obtained  similar large
deviation result  for the number  of isolated  nodes in  the random
graph $\skrig(n,c/n).$

\begin{theorem}[Doku-Amponsah,2006, \,Doku-Amponsah, 2012]\label{main2}
Suppose $D_Y$ is the empirical degree measure of the  random graph
$\skrig(n,nc/2).$
 Then , as $n\to\infty$, $D_Y$ obeys an LDP
in the space $\skriw(\N )$ with good, convex rate function
\begin{equation}\label{randomg.ratedeg}
\begin{aligned}
\delta(d)= \left\{ \begin{array}{ll}H (d\,\|\,q_{c}), & \mbox{if
$\langle d\rangle= c $,}\\

 \infty & \mbox{otherwise.}

\end{array}\right.
\end{aligned}
\end{equation}
where $q_{x}$ is a Poisson distribution with parameter~$x$.
\end{theorem}
Here we remark, that the  LDP result  of (Boucheron  et  al., 2002,
Theorem~7.1) holds  and   the  conjecture that (Boucheron  et  al.,
2002, Lemma~7.2) does not hold is  false. In fact the coupling
argument of (Boucheron  et  al., 2002)
 and Bennett's inequality, see (Bennett, 1962),  proves  Theorem~7.1  of
(Boucheron  et  al., 2002). Recently, (Boucheron  et  al.,
2002,Corollary 1.9) and (Mukherjee, 2013) confirm
 Theorem~\ref{main2}.

Note,  the degree distribution $D_Y$ is a continuous function of
$M_Y,$
and so Theorem~\ref{main1} below and the contraction principle gives the LDP for $D_Y$. 
In fact the LDP for $D_Y$ (above)  is  a  special   case   of
Theorem~\ref{main1} which was first proved  in ( Doku-Amponsah and
Moerters, 2010) or  (Doku-Amponsah, 2012)   by approximating a given
symbolled random graph from below by  another symbolled random
graphs  with the degree of
 each nodes growing  in  order of $o(n^{1/3}).$ See,
( Doku-Amponsah and Moerters, 2010, Lemma~4.10,~p.26-29). Note, that
in special case of classical Erd\H{o}s-R\'enyi graph
$\skrig(n,nc/2)$ we have that $M=D,$ the degree distribution  and
$$\langle \Delta(M_Y)\rangle=2|E|/n= c.$$

\begin{theorem}[ Doku-Amponsah and
Moerters, 2010]\label{randomg.LDprob}\label{main1} Suppose the
sequence $(\nu_n,\pi_n)$ in
$\skriw_n(\skriy)\times\tilde{\skriw}_{n}(\skriy\times\skriy)$
posses a limit $(\nu,\pi)$ in
$\skriw(\skriy)\times\tilde{\skriw}_*(\skriy\times\skriy).$ Let $Y$
be a symbolled random graph with $n$ nodes conditioned on the event
$\{ \Delta(M_Y)=(\nu_n,\pi_n) \}$. Then, as $n\rightarrow\infty,$
the empirical Neighbourhood distribution $M_Y$ obeys  an LDP in
$\skriw(\skriy\times\N^{\skriy})$ with good rate function

\begin{align}\label{randomg.rateLDprob}
\tilde{J}_{(\nu,\pi)}(\mu)=\left\{
\begin{array}{ll}H(\mu\,\|\,Poi) & \mbox
  {if  $(\pi,\mu)$ is sub-consistent  and $\mu_1=\nu$}\\
\infty & \mbox{otherwise,}
\end{array}\right.
\end{align}

$$Poi(a\,,\,l):=\mu_{1}(a)\prod_{b\in\skriy} e^{-\frac{\pi(b,a)}{\mu_1(a)}} \,
\frac{1}{l(b)!}\, \Big(\frac{\pi(b,a)}{\mu_1(a)}\Big)^{l(b)},
\quad\mbox{for $a\in\skriy$, $l\in\N^{\skriy}$} .$$
\end{theorem}

\textbf{3. Proof  of  Main  Results}

\emph{ 3.1 Proof of  Theorem~\ref{main1} : Exponential approximation
by random allocation}

In order to improve( shorten) the proof of
Theorem~\ref{randomg.LDprob}, we pass to a simple \emph{random
allocation model}, which turns out to be equivalent. This model is
best described in term of symbolled balls being placed randomly into
symbolled bins.

Fix $n\ge 1$, a symbol law $\nu_n\in\skriw_n(\skriy)$  and an edge
law $\pi_n\in\tilde{\skriw}_{n}(\skriy\times\skriy).$ The bins
$\skriv=\{1,\ldots, n\}$  are now symbolled by drawing without
replacement from the pool of symbols, which contains the symbol
$a\in\skriy$ exactly $n\nu_n(a)$ times. For each ordered pair
$(b,a)\in\skriy\times\skriy$ of symbols, we independently  and
identically place $nm_n(b,a)$ balls of symbol $b$ into the
$n\nu_n(a)$ bins of symbol $a$ by drawing without replacement. We
denote by $\tilde{\prob}_{(\nu_n,\pi_n)}$ the distribution of the
random allocation model with symbol law $\nu_n\in\skriw_n(\skriy)$
and an edge law $\pi_n\in\tilde{\skriw}_{n}(\skriy\times\skriy)$.

In the resulting constellation we denote, for any bin
$v\in\{1,\ldots, n\}$, by $\tilde{Y}(v)$ its symbol, and by $l^v(b)$
the number of balls of symbol $b\in\skriy$ it contains. Now define
the \emph{empirical occupancy measure} of the constellation by
$$M_{\tilde{Y}}(b,l)
= \frac 1n \sum_{u\in \skriv} \delta_{(\tilde{Y}(u),
\tilde{\skril}(u))}(b, l), \qquad \mbox{ for }
(b,l)\in\skriy\times\N^{\skriy},$$ where
$\tilde{\skril}(u)=(\ell^u(a), a\in\skriy)$ is the symbol
distribution in bin $v$. In our first theorem we establish
exponential equivalence of the law of the empirical  occupancy
measure~$M_{\tilde{Y}}$ under the random allocation model
$\tilde{\prob}_{(\nu_n,\pi_n)}$ and the law of the empirical
Neighbourhood distribution $M$ under
$$\prob_{(\nu_n,\pi_n)}=\prob\{ \,\cdot\,  \,|\,\Delta(M_Y)=(\nu_n,\pi_n)\},$$
the law of the  symbolled  random graph conditioned to have symbol
law $\nu_n$  and edge distribution $\pi_n$. Recall the definition of
exponential equivalence, see (Dembo and Zeitouni, 1998,
Definition~4.2.10).

\begin{lemma}\label{randomg.expequivalnce}\label{exp}
The law of $M_{\tilde{Y}}$ under $\tilde{\prob}_{(\nu_n,\pi_n)}$ is
exponentially equivalent to the law of $M_Y$ under
$\prob_{(\nu_n,\pi_n)}.$
\end{lemma}

\begin{Proof}

Define the metric $d$ of total variation by
$$d(\mu,\tilde{\mu})=\sfrac{1}{2}\sum_{(a,l)\in\skriy\times\N^{\skriy}}
|\mu(a,l)-\tilde{\mu}(a,l)|, \quad \mbox{ for
}\mu,\tilde{\mu}\in\skriw(\skriy\times\N^{\skriy}).$$ As this metric
produces the weak topology, the proof of
Lemma~\ref{randomg.expequivalnce} is equivalent to showing that for
every $\eps>0,$
\begin{equation}\label{randomg.totalv}
\lim_{n\rightarrow\infty}\sfrac{1}{n}\log\prob\big\{d(M_{\tilde{Y}}\,,\,M_Y)\ge\eps\big\}=-\infty,
\end{equation}
where $\prob$ indicates a suitable coupling between the random
allocation model and the symbolled graph.

To  begin,  denote by $V(a)$ the collection of nodes(bins) which
have symbol $a\in\skriy$ and observe that $$\sharp V(a)=n\nu_n(a).$$

For every $a,b\in\skriy$, begin: At each step $k=1,\ldots,
m_n(b,a),$ we randomly pick two nodes $V^k_1\in V(a)$ and $V^k_2\in
V(b)$. Drop one  ball  of  symbol $b$  in  bin  $V^k_1$ and one ball
of symbol $a$ in $V^k_2,$  and  link $V^k_1$ to $V^k_2$ by an edge
unless $V^k_1=V^k_2$ or the two nodes are already linked. If one of
these two things happen, then we simply choose an edge randomly from
the set of all potential edges linking symbols $a$ and $b$, which
are not yet present in the graph. This completes the construction of
a graph with $L_Y^1=\nu_n$, $L_Y^2=\pi_n$ and
\begin{equation}\label{randomg.XTX}
d(M_Y\,,\,M_{\tilde{Y}})\le \sfrac{2}{n}
\sum_{a,b\in\skriy}B^{n}(b,a)\, ,
\end{equation}
where $B^n(b,a)$ is the total number of steps $k\in\{1,\ldots,
m_n(b,a)\}$ at which there is discrepancy between the vertices
$V^k_1$, $V^k_2$ drawn and the nodes which formed the $k^{\rm th}$
edge connecting $a$ and $b$ in the random graph construction.

Given $a,b\in\skriy$,the probability that $V^k_1=V^k_2$ or the two
nodes are already linked is equal to
$$p_{[k]}(b,a):=\sfrac{1}{m_n(b,a)}\1_{\{b=a\}}+\big(1-\sfrac{1}{m_n(b,a)}
\1_{\{b=a\}}\big)\sfrac{(k-1)}{(m_n(b,a))^2}.$$ $B^{n}(b,a)$ is a
sum of independent Bernoulli random variables
$X_1,\,...,\,X_{nm_n(b,a)}$ with `success' probabilities equal to
$p_{[1]}(b,a), \ldots, p_{[nm_n(b,a)]}(b,a)$. Note that $\me[X_k]=
p_{[k]}(b,a)$  and
$$Var[X_k]=p_{[k]}(b,a)(1-p_{[k]}(b,a)).$$  Now,  we   have   $$\me B^{n}(b,a)= \sum_{k=1}^{n(b,a)} p_{[k]}(b,a)=\1_{\{b=a\}}
+\sfrac{1}{2}\big(1-\1_{\{b=a\}}\sfrac{1}{m_n(b,a)}\big)\big(1-\sfrac{1}{m_n(b,a)}\big)\le
\sfrac{1}{2}+\1_{\{b=a\}}.$$

We  write
$$\sigma_n^2(b,a):=\sfrac{1}{m_n(b,a)}\sum_{k=1}^{m_n(b,a)}Var[X_k]$$
and  observe  that   $$\lim_{n\to \infty}\me(B^n(b,a))=\lim_{n\to
\infty}Var(B^n(b,a))=\lim_{n\to
\infty}m_n(b,a)\sigma_n^2(b,a)=\1_{\{b=a\}}+\sfrac{1}{2}.$$

We  Define  $h(x)=(1+x)\log(1+x)-x,$ for  $x\ge 0$   and use
Bennett's inequality, see ( Bennett,  1962), to  obtain, for
sufficiently large $n$
$$\P\big\{ \sfrac{ B^{n}(b,a)}{n}\ge\sfrac{ \1_{\{b=a\}}+\sfrac{1}{2}}{n}+\delta_{1}\big\}
\le
exp\Big[-m_n(b,a)\sigma_n^2(b,a)h(\sfrac{n\delta_{1}}{n(b,a)\sigma_n^2(b,a)})\Big],$$
for any $\delta_1>0.$ Let $\eps\ge 0 $ and  choose
$\delta_1=\sfrac{\eps}{2m^2}.$ Suppose that we have
$B^n(b,a)\le\delta$. Then, by~\eqref{randomg.XTX},
$$d(M_Y,{\mu}_n)\le 2\delta_1 m^2=\eps.$$ Hence,
$$\begin{aligned}
\prob\big\{ d(M_Y,M_{\tilde{Y}}) > \eps \big\}  \le
\sum_{a,b\in\skriy} \prob\big\{ B^n(b,a)\ge n\delta_1 \big\}&\le
m^2\sup_{a,b\in\skriy}\prob\big\{ B^n(b,a)\ge
\1_{\{b=a\}}+\sfrac{1}{2}+ (n\delta_1)/2
\big\}\\
 & \le m^2\sup_{a,b\in\skriy} exp\Big[-m_n(b,a)\sigma_n^2(b,a)h(\sfrac{n\delta_{1}}{m_n(b,a)\sigma_n^2(b,a)})\Big] .
\end{aligned}$$

Let   $0\le \delta_2\le 1$. The,  for  sufficiently  large $n$ we
 have
\begin{equation}\begin{aligned}\label{Equ.coupling}
\frac{1}{n}&
\log\P\Big\{d(M_Y,M_{\tilde{Y}}) > \eps \Big\}\le-(1-\delta_{2})h(\sfrac{n\delta_1}{2(1+\delta_{2})})\\
&=-(\1_{\{b=a\}}+\sfrac{1}{2}-\delta_{2})\Big[(\sfrac{1}{n}+\sfrac{\delta_1}{2(\1_{\{b=a\}}+\sfrac{1}{2}+\delta_{2})})\log(1+\sfrac{n\delta_1}{2(\1_{\{b=a\}}
+\sfrac{1}{2}+\delta_{2})})-\sfrac{\delta_1}{2(\1_{\{b=a\}}+\sfrac{1}{2}+\delta_{2})}\Big].
\end{aligned}\end{equation}

This completes the proof of the lemma.
\end{Proof}

\emph{3.2 Proof of  Theorem~\ref{main1}: Large  deviation
probabilities by the  method  of  types}

Now Lemma~\ref{randomg.expequivalnce} and the large deviation
principle for $M_{\tilde{Y}}$ under $\tilde{\prob}_{(\nu_n,\pi_n)}$
implies the same large deviation principle for $M_Y$ under
$\prob_{(\nu_n,\pi_n)}$ in the weak topology. See,for  example (
 Dembo and  Zeitouni, 1998, Theorem~4.2.13). Consequently, the proof of
Theorem~\ref{randomg.LDprob} is equivalent to showing that for every
$\Sigma\subset\skriw(\skriy\times\N^{\skriy}),$
\begin{lemma}\label{main11}
\begin{align}
-\inf_{\mu\in
int(\Sigma)}\tilde{J}_{(\nu,\pi)}(\mu)&\le\liminf_{n\rightarrow\infty}\sfrac1n\log\tilde{\prob}\Big\{M_{\tilde{Y}}\in\Sigma\,\big|\,
\Delta(M_{\tilde{Y}})=(\nu_n,\pi_n)\Big\}\nonumber\\
&\le\limsup_{n\rightarrow\infty}\sfrac1n\log\tilde{\prob}\Big\{M_{\tilde{Y}}\in\Sigma\,\big|\,\Delta(M_{\tilde{Y}})=(\nu_n,\pi_n)\Big\}\le-\inf_{\mu\in
cl(\Sigma)}\tilde{J}_{(\nu,\pi)}(\mu).\nonumber
\end{align}
\end{lemma}

We begin the proof of Lemma~\ref{main11} by recalling the definition
of $\Delta$ a function on $\skriw(\skriy\times\N^{\skriy})$ given by
$$\mu\mapsto(\Delta_1(\mu),\Delta_2(\mu)).$$ Let
$(\nu_n,\pi_n)\rightarrow(\nu,\pi)\in\skriw(\skriy)\times\skriw(\skriy\times\skriy)$
and
 write  $$
\skrik^{(n)}(\nu_n,\pi_n)=\big\{\mu_n:\mu_n=M,~\Delta(\mu_n)=(\nu_n,\pi_n),~for~some~random~allocations
~process~on~ n~ bins~\big\}.$$

We denote by $\skris(\mu)$ the support of $ \mu$ and write for
$\mu_n\in\skrik^{(n)}(\nu_n,\pi_n),$

$$\vartheta_{1}^{(n)}(\pi_n,\mu_n)=n\alpha_1^{(n)}(\nu_n,\pi_n)-n\beta_1^{(n)}(\mu_n)-\sfrac{1}{2n}|\skris(\mu_n)|\log2\pi
n,$$  where
\begin{align*}
\alpha_1^{(n)}(\nu_n,\pi_n)=-\sfrac{1}{n}\log|\skrik^{(n)}(\nu_n,\pi_n)|&+\sfrac{1}{n}\sum_{a,b\in\skriy}\log\pi_n(b,a)+\sfrac{1}{2n}\big(|\skriy|+|\skriy|^{2}\big)\log2\pi n\\
&+\sfrac{1}{n^2}\sum_{a\in\skriy}\sfrac{1}{12\nu_n(a)+1/n}+\sfrac{1}{n}\sum_{a\in\skriy}\log\nu_n(a)+\sfrac{1}{n^2}\sum_{a,b\in\skriy}\sfrac{1}{12\pi_n(b,a)+1/n},
\end{align*}

$$\beta_1^{(n)}(\mu_n)=\sfrac{1}{n}\sum_{\heap{(a,l)\in\skriy\times\N^{\skriy}}{\mu_n(a,l)>0}}\log\mu_n(a,l)+\sfrac{1}{n^2}\sum_{\heap{(a,l)\in\skriy\times\N^{\skriy}}{\mu_n(a,l)>0}}\sfrac{1}{12\mu_n(a,l)+1/n}.$$
We write
$\vartheta_{2}^{(n)}(\pi_n,\mu_n)=n\alpha_2^{(n)}(\nu_n,\pi_n)-n\beta_2^{(n)}(\mu_n),$
where

\begin{align*}
\alpha_2^{(n)}(\nu_n,\pi_n)=\sfrac{1}{n}\log|\skrik^{(n)}(\nu_n,\pi_n)|+\sfrac{1}{n^2}\sum_{a,b\in\skriy}\sfrac{1}{12\pi_n(b,a)}&+\sfrac{1}{n^2}\sum_{\heap{a\in\skriy}{\nu_n(a)>0}}\sfrac{1}{12\nu_n(a)}+\sfrac{1}{n}\sum_{a,b\in\skriy}\log\pi_n(b,a)\\
&+\sfrac{1}{2n}\big(|\skriy|+|\skriy|^{2})\log2\pi
n+\sfrac{1}{n}\sum_{a\in\skriy}\log\nu_n(a),
\end{align*}
$$\beta_2^{(n)}(\mu_n)=\sfrac{1}{n}\sum_{\heap{(a,l)\in\skriy\times\N^{\skriy}}{\mu_n(a,l)>0}}\log\mu_n(a,l)\big)+\sfrac{1}{n}\sum_{\heap{(a,l)\in\skriy\times\N^{\skriy}}{\mu_n(a,l)>0}}\sfrac{1}{12\mu_n(a,l)}$$

 We prove Lemma~\ref{main11} above  from  the  following lemma
 which uses the  the  idea of  the method  of  types. see,
 (
 Dembo and  Zeitouni, 1998, Chapter~2).

\begin{lemma}\label{randomg.LDprobm}
For any  $\mu_n\in\skrik^{(n)}(\nu_n,\pi_n),$
\begin{equation}\nonumber
e^{-nH(\mu_n\,\|\,Poi_n)+\vartheta_{1}^{(n)}(\pi_n,\mu_n)} \le
\tilde{\prob}\big\{M_{\tilde{Y}}=\mu_n\,\big|\,\Delta(M_{\tilde{Y}})=(\nu_n,\pi_n)\big\}\le|\skrik^{(n)}(\nu_n,\pi_n)|^{-1}
e^{-nH(\mu_n\,\|\,Poi_n)+\vartheta_{2}^{(n)}(\pi_n,\mu_n)},\end{equation}
where
$$Poi_n(a\,,\,l)=\nu_n(a)\prod_{b\in\skriy}\frac{e^{-\pi_n(b,a)/\nu_n(a)}[\pi_n(b,a)/\nu_n(a)]^{l(b)}}{l(b)!},\,\mbox{for
$l\in\N^{\skriy}$} $$

$$\lim_{n\to \infty}\vartheta_{2}^{(n)}(\pi_n,\mu_n)=\lim_{n\to
\infty}\vartheta_{1}^{(n)}(\pi_n,\mu_n)=0$$

\end{lemma}
\begin{proof}
The proof of this lemma uses the  idea of the method of types, see,
 (Dembo and  Zeitouni, 1998, Chapter~2), combinatoric argument and
good estimates from  refined Stirling's formula. \footnote[1]
{$(2\pi)^{\sfrac12}n^{n+\sfrac12}e^{-n+1/(12n+1)}<n!<(2\pi)^{\sfrac12}n^{n+\sfrac12}e^{-n+1/(12n)},$
see ( Feller, 1971, page~52).}
\\ We denote by $\tilde{Y}$ the random allocation  process  and  observe that for any $\mu_n\in\skrik^{(n)}(\nu_n,\pi_n)$
we have
\begin{equation}\label{randomg.probcom}
\tilde{\prob}\big\{M_{\tilde{Y}}=\mu_n\,\big|\,\Delta(M_{\tilde{Y}})
=(\nu_n,\pi_n)\big\}=\frac{\sharp\big\{\tilde{Y}:\,M_{\tilde{Y}}=\mu_n,\,\Delta(M_{\tilde{Y}})
=(\nu_n,\pi_n)
\big\}}{\sharp\big\{\tilde{Y}:\,(L_{\tilde{Y}}^{1},L_{\tilde{Y}}^2)=(\nu_n,\pi_n)\big\}}.
\end{equation}

Now, the right side of \eqref{randomg.probcom} may be evaluated in
the following way:

\begin{itemize}
\item  For a given  empirical  measure $\mu_n$  with  $\Delta(\mu_n)=(\nu_n,\pi_n) $ there are
$$\prod_{a\in\skriy}\Big(\heap{n\nu_n(a)}{n\mu_n(a,l),\,l\in\N^{\skriy}}\Big)\prod_{a,b\in\skriy}\Big(\heap{n\pi_n(b,a)}{l_{a}^{(j)}(b),\,j=1,...,n\nu_n(a)}\Big)$$
 equally likely   random allocation processes  and
\item for every  empirical and empirical pair measure
$\Delta(\mu_n)=(\nu_n,\pi_n)$  there  are
 $\prod_{a,b\in\skriy}\Big(n\nu_n(a)\Big)^{n\pi_n(b,a)}$
 equally likely   random allocation processes
\end{itemize}
 Therefore,
\eqref{randomg.probcom} is equivalent to
\begin{align}
\tilde{\prob}\big\{\tilde{M}&=\mu_n\,\big|\,\Delta(\tilde{M})=(\nu_n,\pi_n)\big\}\\
&=\prod_{a\in\skriy}\Big(\heap{n\nu_n(a)}{n\mu_n(a,l),\,l\in\N^{\skriy}}\Big)\prod_{a,b\in\skriy}\Big(\heap{n\pi_n(b,a)}{l_{a}^{(j)}(b),\,j=1,...,n\nu_n(a)}\Big)\Big(\frac{1}{n\nu_n(a)}\Big)^{n\pi_n(b,a)},\label{Tclass}
\end{align}
while $\tilde{\prob}\big\{\tilde{M}=\mu_n\,\big|\,\Delta(\tilde{M})
=(\nu_n,\pi_n)\big\}=0$ when $\Delta(\mu_n)
\not=(\nu_n,\pi_n)$ by convention.\\

Suppose $\pi_n(b,a)=0,$  for some $a,b\in\skriy$ then

\begin{equation}\label{randomg.One}
\Big(\heap{n\pi_n(b,a)}{l_{a}^{(j)}(b),\,j=1,...,n\nu_n(a)}\Big)=1.
\end{equation}

Suppose  $\pi_n(b,a)>0,$  a good estimate of $(n\pi_n(b,a))!$ can be
obtained  from the refined Stirling's approximation, as

\begin{align*}
\exp\Big(n\pi_n(b,a)&\log
n\pi_n(b,a)-n\pi_n(b,a)+\sfrac12\log\pi_n(b,a)
+\sfrac{1}{12n\pi_n(b,a)+1}+\sfrac12\log
2n\pi\Big)\le(n\pi_n(b,a))!\\
&\le\exp\Big(n\pi_n(b,a)\log
n\pi_n(b,a)-n\pi_n(b,a)+\sfrac12\log\pi_n(b,a)+\sfrac{1}{12n\pi_n(b,a)}+\sfrac{1}{2}\log2\pi
n).
\end{align*}

Similarly, from the refined Stirling's approximation,  see
 ( Feller, 1971, page~52), we have

\begin{align}
&\exp\Big(n\sum_{a\in\skriy}\nu_n(a)\log\nu_n(a)-n\sum_{(a,l)}\mu_n(a,l)\log\mu_n(a,l)+\sfrac{1}{n}\sum_{a\in\skriy}\sfrac{1}{12\nu_n(a)+1/n}+\sfrac{|\skriy|-|\skris(\mu_n)|}{2}\log2\pi
n\Big)\nonumber\\
&\times\exp\Big(-\sfrac1n\sum_{(a,l)\in\skriy\times\N^{\skriy}\atop\mu_n(a,l)>0}\sfrac{1}{12\mu_n(a,l)+1/n}\Big)\le\prod_{a\in\skriy}\Big(\heap{n\nu_n(a)}{n\mu_n(a,l),\,l\in\N^{\skriy}}\Big)\le\exp\Big(n\sum_{a\in\skriy}\nu_n(a)\log\nu_n(a)\Big)\nonumber\\
&\times\exp\Big(-n\sum_{(a,l)}\mu_n(a,l)\log\mu_n(a,l)-\sfrac1n\sum_{(a,l)\in\skriy\times\N^{\skriy}\atop\mu_n(a,l)>0}\sfrac{1}{12\mu_n(a,l)}+\sfrac{|\skriy|-|\skris(\mu_n)|}{2}\log2\pi
n+\sfrac{1}{n}\sum_{a\in\skriy}\sfrac{1}{12\nu_n(a)}\Big).\label{randomg.condientrop}
\end{align}

We observe that
 $\prod_{a,b\in\skriy}\prod_{j=1}^{n\nu_n(a)}l_{a}^{(j)}(b)!=\prod_{b\in\skriy}\exp\big(n\sum_{(a,l)}(\log l(b)!)\mu_n(b,a)\big),$
and hence
\begin{align*}
&\exp\Big(\sum_{b\in\skriy}\big[n\sum_{a\in\skriy}\pi_n(b,a)\log
\pi_n(b,a)-n\sum_{a\in\skriy}\pi_n(b,a)-n\sum_{a\in\skriy}\pi_n(b,a)\log\nu_n(a)\big]+\sfrac12\sum_{a,b\in\skriy}\log\pi_n(b,a)\Big)\\
&\times\exp\Big(n\sum_{b\in\skriy}\sum_{(a,l)}(\log
l(b)!)\mu_n(a,l)+\sfrac{|\skriy|^2}{2}\log2\pi
n+\sum_{a,b\in\skriy}\sfrac{1}{12n\pi_n(b,a)+1}\Big)\\
&\le\prod_{a,b\in\skriy}\Big(\heap{n\pi_n(b,a)}{l_{a}^{(j)}(b),\,j=1,...,n\nu_n(a)}\Big)\Big(\frac{1}{n\nu_n(a)}\Big)^{n\pi_n(b,a)}\\
&\le\exp\Big(n\sum_{b\in\skriy}\sum_{(a,l)}(\log l(b)!)\mu_n(a,l)+\sfrac12\sum_{a,b\in\skriy}\log\pi_n(b,a)\Big)\\
&\times\exp\Big(\sum_{b\in\skriy}\big[n\sum_{a\in\skriy}\pi_n(b,a)\log\sum_{a\in\skriy}\pi_n(b,a)-n\sum_{a\in\skriy}\pi_n(b,a)
-n\sum_{a\in\skriy}\pi_n(b,a)\log\nu_n(a)\big]+\sum_{a,b\in\skriy}\sfrac{1}{12n\pi_n(b,a)}\Big)\\
&\times\exp\Big(\sfrac{|\skriy|^2}{2}\log2\pi n\Big).
\end{align*}

Putting everything together and choosing
$\vartheta_{1}^{(n)}(\pi_n,\mu_n)$ and
$\vartheta_{2}^{(n)}(\pi_n,\mu_n)$ appropriately,  we have that
\begin{align*}
& \exp\Big(nH(\mu_n)+\sum_{b\in\skriy}\big[n\sum_{a\in\skriy}\pi_n(b,a)\log\sum_{a\in\skriy}\pi_n(b,a)-n\sum_{a\in\skriy}\pi_n(b,a)-n\sum_{a\in\skriy}\pi_n(b,a)\log\nu_n(a)\big]\Big)\\
&\times\exp\Big(-nH(\nu)-n\sum_{b\in\skriy}\sum_{(a,l)}(\log l(b)!)\mu_n(a,l)+\vartheta_{1}^{(n)}(\pi_n,\mu_n)\Big)\le\tilde{\prob}\big\{M=\mu_n\,\big|\,\Delta(M)=(\nu_n,\pi_n)\big\}\\
&\le \exp\Big(-nH(\nu)-\sum_{b\in\skriy}\big[n\sum_{(a,l)}(\log l(b)!)\mu_n(a,l)-n\sum_{a\in\skriy}\pi_n(b,a)-n\sum_{a\in\skriy}\pi_n(b,a)\log\nu_n(a)\big]\Big)\\
&\times\exp\Big(nH(\mu_n)+n\sum_{b\in\skriy}\sum_{a\in\skriy}\pi_n(b,a)\log\sum_{a\in\skriy}\pi_n(b,a)-\log|\skrik^{(n)}(\nu_n,\pi_n)|+\vartheta_{2}^{(n)}(\pi_n,\mu_n)\Big).
\end{align*}
Collecting and rearranging terms properly and using
$\Delta(\mu_n)=(\nu_n,\pi_n),$  we have that
\begin{align*}
&H(\nu_n)-H(\mu_n)-\sum_{b\in\skriy}\Big[\sum_{(a,l)}l(b)\mu_n(a,l)\log\sum_{(a,l)}l(b)\mu_n(a,l)-\sum_{(a,l)}l(b)\mu_n(a,l)-\sum_{(a,l)}l(b)\mu_n(a,l)\log\nu_n(a)\\
&-\sum_{(a,l)}(\log l(b)!)\mu_n(a,l)\Big]\\
&=\sum_{(a,l)}\mu_n(a,l)\Big[\log\mu_n(a,l)-\log\nu_n(a)-\sum_{b\in\skriy}\Big(\log\big(\sfrac{\pi_n(b,a)}{\nu_n(a)})^{l(b)}-\sfrac{\pi_n(b,a)}{\nu_n(a)}-\log\sum_{l}(\log l(b)!)\Big)\Big]\\
&=\sum_{(a,l)}\mu_n(a,l)\Big[\log\mu_n(a,l)-\log\big(\nu_n(a)\prod_{b\in\skriy}\sfrac{(\pi_n(b,a)/\nu_n(a))^{l(b)}\exp(-\pi_n(b,a)/\nu_n(a))}{l(b)!}\big)\Big]\\
&=H(\mu_n\,\|\,Poi_n)
\end{align*}
which completes the proof of Lemma~\ref{randomg.LDprobm}.\end{proof}
We prove  from Lemma~\ref{randomg.LDprobm} and ( Doku-Amponsah
 and  Moerters, 2010, Lemmas~4.1~and~4.4), upper bounds and lower bounds in the
large deviation principle for all finite $n.$ Let
$\Sigma\subset\skriw(\skriy\times\N^{\skriy}).$ Then
\eqref{randomg.LDprobm} gives the upper bound
\begin{align}
\tilde{\prob}\Big\{M_{\tilde{Y}}\in\Sigma\,\big|\,\Delta(M_{\tilde{Y}})=(\nu_n,\pi_n)\Big\}&
=\sum_{\mu_n\in\Sigma\cap\skrik^{(n)}(\nu_n,\pi_n)}\tilde{\prob}\big\{M_{\tilde{Y}}=\mu_n\,\big|\,\Delta(M_{\tilde{Y}})=(\nu_n,\pi_n)\big\}\nonumber\\
&\le\sum_{\mu_n\in\Sigma\cap\skrik^{(n)}(\nu_n,\pi_n)}|\skrik^{(n)}(\nu_n,\pi_n)|^{-1}e^{-nH(\mu_n\,\|\,Poi_n)+n\vartheta_2^{(n)}(\nu_n,\pi_n)}\nonumber\\
&\le
e^{-n\inf_{\mu_n\in\Sigma\cap\skrik^{(n)}(\nu_n,\pi_n)}H(\mu_n\,\|\,Poi_n)+n\vartheta_2^{(n)}(\nu_n,\pi_n)}.\label{randomg.probupper}
\end{align}
The corresponding lower bound is
\begin{align}
\tilde{\prob}\Big\{M_{\tilde{Y}}\in\Sigma\,\big|\,\Delta(M_{\tilde{Y}})&=(\nu_n,\pi_n)\Big\}
=\sum_{\mu_n\in\Sigma\cap\skrik^{(n)}(\nu_n,\pi_n)}\tilde{\prob}\big\{M_{\tilde{Y}}=\mu_n\,\big|\,\Delta(M_{\tilde{Y}})=(\nu_n,\pi_n)\big\}\nonumber\\
&\ge
e^{-n\inf_{\mu_n\in\Sigma\cap\skrik^{(n)}(\nu_n,\pi_n)}H(\mu_n\,\|\,Poi_n)}\sum_{\mu_n\in\Sigma\cap\skrik^{(n)}
(\nu_n,\pi_n)}e^{\vartheta_1^{(n)}(\nu_n,\pi_n)}\nonumber\\
&\ge
e^{-n\inf_{\mu_n\in\Sigma\cap\skrik^{(n)}(\nu_n,\pi_n)}H(\mu_n\,\|\,Poi_n)+n\vartheta_1^{(n)}(\nu_n,\pi_n)}.\label{randomg.problower}
\end{align}

Since
$\limsup_{n\rightarrow\infty}\sfrac{1}{n}\vartheta_{2}^{(n)}(\nu_n,\pi_n)=0
$ and
$\liminf_{n\rightarrow\infty}\sfrac{1}{n}\vartheta_{1}^{(n)}(\nu_n,\pi_n)=0$
( by  (Doku-Amponsah and  Moerters, 2010,Lemmas~4.1~and~4.4)), the
normalized logarithmic limits of \eqref{randomg.probupper} and
\eqref{randomg.problower} gives
\begin{equation}
\limsup_{n\rightarrow\infty}\sfrac{1}{n}\log\tilde{\prob}\Big\{M_{\tilde{Y}}\in\Sigma\,\big|\,\Delta(M_{\tilde{Y}})=(\nu_n,\pi_n)\Big\}
=-\liminf_{n\rightarrow\infty}\big\{\inf_{\mu_n\in\Sigma\cap\skrik^{(n)}(\nu_n,\pi_n)}H(\mu_n\,\|\,Poi_n)\big\}\label{randomg.puu}
\end{equation}
and

\begin{equation}
\liminf_{n\rightarrow\infty}\sfrac{1}{n}\log\tilde{\prob}\Big\{M_{\tilde{Y}}\in\Sigma\,\big|\,\Delta(M_{\tilde{Y}})=(\nu_n,\pi_n)\Big\}=-\limsup_{n\rightarrow\infty}\big\{\inf_{\mu_n\in\Sigma\cap\skrik^{(n)}(\nu_n,\pi_n)}H(\mu_n\,\|\,Poi_n)\big\}\label{randomg.pll}
\end{equation}
The upper bound in \eqref{randomg.LDprobm} follows from
\eqref{randomg.puu} , as
$\Sigma\cap\skrik^{(n)}(\nu_n,\pi_n)\subset\Sigma$  for all $n.$

Now  fix $\mu\in \skriw(\skriy\times\N^{\skriy}).$ Then, by \cite
[Lemma~4.9]{DM06a}, there
 exists a sequence
 $\mu_n\in\Sigma\cap\skrik^{(n)}(\nu_n,\pi_n)$ such that $\mu_n\rightarrow\mu$ as $n\rightarrow\infty.$
 Therefore, by continuity of  entropy,  see, example \cite[p.19,~equation~14]{DM06a} we have  that
\begin{equation}\nonumber
\limsup_{n\rightarrow\infty}\big\{\inf_{\mu^{'}\in\Sigma\cap\skrik^{(n)}(\nu_n,\pi_n)}H(\mu^{'}\,\|\,Poi_n)\big\}\le
\lim_{n\rightarrow\infty}H(\mu_n\,\|\,Poi_n)=H(\mu\,\|\,Poi).
\end{equation}
Recall that $H(\mu\,\|\,Q)=\infty$ whenever, for some
$(b,l)\in\skriy\times\N^{\skriy},$ $\mu(b,l)>0$ while $Poi(b,l)=0.$
Hence, by the preceding inequality  we have
  $$\limsup_{n\rightarrow\infty}\big\{\inf_{\mu^{'}\in\Sigma\cap\skrik^{(n)}(\nu_n,\pi_n)}H(\mu^{'}\,\|\,Poi_n)\big\}\le
\inf_{\mu\in int(\Sigma)}H(\mu\,\|\,Poi),$$ which gives the lower
bound in \eqref{randomg.LDprobm}, for $\mu$ satisfying
$\Delta(\mu)=(\nu,\pi).$ To conclude the prove of the lower bound we
note that  by \cite[Lemma~4.6]{DM06a} for any $\mu\in\Sigma$ with
$\mu_1=\nu$ and $\Delta_2(\mu)\le\pi$ there exists $\mu_n\in
\skrik^{(n)}(\nu_n,\pi_n)$ converging  weakly to $\mu$  such that
$H(\mu_n\,\|\,Poi_n)$ converges to $H(\mu\,\|\,Poi).$

\emph{Proof  of  Theorem~\ref{main3}}

We  obtain  this corollary from Theorem~\ref{main2}  by the
application of  the contraction principle, ( Dembo and
 Zeitouni, 1998, Theorem~4.2.1) to the linear map $F:\skriw(\N)\to [0,1]$
given  by $F(d)=d(0).$

In fact  Theorem~\ref{main2}  implies  an  LDP  for  random variable
$F(D_Y)=D_Y(0)$ with  good,  convex  rate  function
$$\eta(x)=\inf\big\{H(d\,\|\,q_c):d\in\skriw(\N), d(0)=x, \,\sum_{k=0}^{\infty}kd(k)= c\,\big\}.$$

Note that, for  a  general $x$ the  class of  distributions
satisfying the two constraints  might  be  non  empty. Since  we
have  $$c=\sum_{k=1}^{\infty}kd(k)\ge \sum_{k=1}^{\infty}d(k)=1-x,$$
the  class  is  necessarily  empty if  $c<1-x$.  If  $c\ge 1-x$,  a
Lagrangian  calculation  gives that  the  mininum  is  attained at
$p,$  defined  by $ p(0)=x$,
$p(k):=Z(x,c)^{-1}\sfrac{(\lambda(x,c))^k}{k!}$ where $\lambda(x,c)$
is  the  unique  root  of
$$\sfrac{e^{\lambda}-1}{\lambda}=\sfrac{1-x}{c}$$
and  $Z(x,c):=\sfrac{e^{\lambda-1}}{1-x}.$  Therefore we  have that

\begin{equation}\begin{aligned}\label{equ.last}
\eta(x)&=x\log\frac{x}{q_c(0)}
+(1-x)\log\frac{(1-x)}{1-q_c(0)}+(1-x)\sum_{k=1}^{\infty}d_x(k)\log\sfrac{d_x(k)}{\hat{q}_{c}(k)}.\\
&=x\log\frac{x}{q_c(0)} +(1-x)\log\frac{(1-x)}{1-q_c(0)}+c\log
\sfrac{\lambda}{c}
\end{aligned}
\end{equation}
if  $c\ge 1-x$  and  $\infty$  otherwise. In  particular if
$x=e^{-c}$ then  $\lambda(x,c)=c,$  which gives  $\eta(e^{-c})=0.$
This  completes  the  proof  of  the theorem.

\textbf{Acknowledgements}

We  are  thankful  to  the  referees  for their suggestions which
have helped  improved  this article.

\textbf{References}

\hangafter=1

\setlength{\hangindent}{2em}

{\sc Bordenave,C.} and {\sc Caputo,  P.}(2013). {\newblock Large
deviations of empirical neighborhood distribution in sparse random
graphs.} {\newblock arxiv:1308.5725 (2013).}

\hangafter=1

\setlength{\hangindent}{2em}

\ {\sc Bennett, G.}(1962)
\newblock {Probability Inequalities for the Sum of
Independent Random Variables}
\newblock {\emph{Journal of the American Statistical
Association 57 (297): 33–45. doi:10.2307/2282438 (1962)}}

\hangafter=1

\setlength{\hangindent}{2em}

 {\sc Boucheron, S.} {\sc Gamboa, F. }and {\sc Leonard, C.}(2002).
 \newblock{ Bins and balls: Large deviations of the empirical occupancy process.}
\newblock{\emph{Ann. Appl. Probab. 12 607-636 (2002).}http://dx.doi.org/10.1214/aoap/1026915618}

\hangafter=1

\setlength{\hangindent}{2em}

{\sc  Biggins, J.D.}(2004).
\newblock{Large deviations for mixtures.}
\newblock{\emph{ El. Comm. Probab.9 60 71 (2004).}}

\hangafter=1

\setlength{\hangindent}{2em}

 {\sc Biggins, J.D.} ~and~ {\sc Penman, D. B.}(2009)
\newblock{ Large deviations in randomly coloured random graphs.}
\newblock{Electron. Comm. Probab. 14 290–301(2009).}
\newblock{ Mathematical Reviews (MathSciNet): MR2524980  Zentralblatt MATH: 1185.05125.}
\hangafter=1

\setlength{\hangindent}{2em}

{\sc Cannings, C.} and {\sc ~Penman,D.B}(2003)
\newblock{Models of random graphs and their applications}
\newblock{\em Handbook of Statistics 21. Stochastic Processes: Modeling
and Simulation.}
\newblock{Eds:\,D.N. Shanbhag and C.R. Rao}. Elsevier (2003) 51-91.

\hangafter=1

\setlength{\hangindent}{2em}

{\sc Doku-Amponsah, K.}(2006).
\newblock{Large deviations and basic information theory for hierarchical and networked data structures.}
\newblock PhD Thesis, Bath (2006).

\hangafter=1

\setlength{\hangindent}{2em}

{\sc Doku-Amponsah,  K.}(2012).
\newblock{Basic information theory, thermo-limits  for network structures.}
\newblock LAP Lambert Academic Publishing, (2012).

\hangafter=1

\setlength{\hangindent}{2em}

{\sc Doku-Amponsah,  K.}(2011)
\newblock{Asymptotic equipartition properties for hierarchical and networked  structures.}
\newblock ESAIM:\emph{Probability  and Statistics}.DOI: 10.1051/ps/2010016 :
Published online by Cambridge University Press: 03 February
2011.http://dx.doi.org/10.1051/ps/2010016.

\hangafter=1

\setlength{\hangindent}{2em}

{\sc Doku-Amponsah, K.}  and {\sc M\"orters, P.}(2010)
\newblock{Large deviation principle for  empirical measures of
coloured random graphs.}
\newblock{ \emph{ The  annals  of  Applied  Probability,} 20,6 (2010),1989-2021. http://dx.doi.org/10.1214/09-AAP647.}

{\sc Dembo, A., M\"orters, P.} and {\sc Sheffield, S.}(2005).
\newblock Large deviations of Markov chains indexed by random trees.
\newblock \emph{Ann. Inst. Henri Poincar\'e: Probab.et Stat.41,}
(2005) 971-996.http://dx.doi.org/10.1016/j.anihpb.2004.09.005.

\hangafter=1

\setlength{\hangindent}{2em}

{\sc Dembo, A.} and {\sc Zeitouni,O. }(1998).
\newblock Large deviations techniques and applications.
\newblock Springer, New York, (1998).

\hangafter=1

\setlength{\hangindent}{2em}

{\sc Feller, W.}(1967)
\newblock{An introduction to probability theory and its applications.}
\newblock{Vol.~I, Wiley, New York. Third edition, (1967).}
\hangafter=1

\setlength{\hangindent}{2em}

{\sc Van Der Hofstad, R.}(2009).
\newblock {Random  Graphs  and  Complex
Networks.}
\newblock{ Eindhoven  University   of  Technology.}
\newblock  {Unpublish  Manuscript.}

\hangafter=1

\setlength{\hangindent}{2em}

{\sc  Mukherjee, S.}(2013).
\newblock {Large deviation for the empirical
degree distribution of an Erdos-Renyi graph.}
\newblock{ arXiv:1310.4160 (2013).}

\hangafter=1

\setlength{\hangindent}{2em}

 {\sc ~Newman, m.e.}(2000)
 \newblock Random graphs as models of networks.
\newblock {\emph{http://arxiv.org/abs/cond-mat/0202208}}

\hangafter=1

\setlength{\hangindent}{2em}

{\sc N. O'Connell,  N.} (1998).
\newblock Some large deviation results for sparse random graphs.
\newblock \emph{Probab. Theory Relat. Fields} {\bf 110} 277--285 (1998).
\smallskip

\hangafter=1

\setlength{\hangindent}{2em}

{\sc Penman, D.B.}(1998).
\newblock{Random graphs with correlation structure.}
\newblock PhD Thesis, Sheffield 1998.
\smallskip

{\sc Sion, M.}(1958).
\newblock{On  general  Minimax theorems.}
\newblock {Pacific  J. Math. 8 171-176 1958.}
\smallskip

\vspace{0.5cm}
\textbf{Copyrights}

Copyright for this article is retained by the author(s), with first publication rights granted to the journal.

This is an open-access article distributed under the terms and conditions of the Creative Commons Attribution license (http://creativecommons.org/licenses/by/3.0/).

\end{document}